\documentclass[a4paper,8pt]{article}
\usepackage[english]{babel}
\usepackage{geometry}
\usepackage[T1]{fontenc}
\usepackage[latin1]{inputenc}
\usepackage{amsmath,amssymb,mathrsfs,amsthm}
\usepackage{soul}
\usepackage{epsfig}
\usepackage{hyperref}
\usepackage{graphicx}
\usepackage{xypic}
\usepackage{datetime}
\usepackage{fancyhdr}
\fancypagestyle{plain}{
\lhead{}
\rhead{}
\lfoot{Date \today{} at \currenttime}

}

\usepackage{tocloft}

\newtheorem{theorem}{Theorem}[section]

\newtheorem{claim}[theorem]{Claim}

\newtheorem{definition}[theorem]{Definition}

\newcommand{\vc}{\|\cdot\|}
\newcommand{\C}{\mathbb{C}}

\newcommand{\Z}{\mathbb{Z}}

\newcommand{\s}{\mathbb{S}^1}

\newcommand{\R}{\mathbb{R}}

\newcommand{\p}{\mathbb{P}}

\newcommand{\Q}{\mathbb{Q}}
\newcommand{\N}{\mathbb{N}}

\title{On an arithmetic inequality on $\p^1_\Q$}
\date{\today, \currenttime}
\author{Mounir Hajli\footnote{\small National Center for Theoretical Sciences (Taipei Office) National Taiwan University,
Taipei 106, Taiwan \quad\quad\quad \emph{E-mail}:\ttfamily{hajlimounir@gmail.com}}}
\begin{document}

\maketitle
\begin{abstract}
We establish an inequality comparing the height and the $\chi$-arithmetic volume of toric metrized divisors on
$\p^1_\Q$. This gives a partial
 answer to a question of Burgos, Moriwaki, Philippon and Sombra 
(\cite[remark 5.13]{Burgos3}).\\
\end{abstract}

In \cite[remark 5.13]{Burgos3} the authors ask if the following inequality
\[
h_{\overline{D}}(X)\leq \widehat{\mathrm{vol}}_\chi(X,\overline{D})
\]
holds for  any toric DSP metrized $\R$-divisor $\overline{D}
$ on $X=\p^1_\mathbb{K}$,  where  $\mathbb{K}$  
 is a global field, $h_{\overline{D}}(X)$ is
the height of $X$  and $\widehat{\mathrm{vol}}_\chi(X,\overline{D})$ is
$\chi$-arithmetic volume with respect to $\overline{D}$.

 In this note we give an affirmative answer to this question when
$\mathbb{K}=\Q$,
 $D$ is nef and $\overline{D}$ is a toric DSP metrized 
divisor  such that the metric  on all non-archimedean
places is the canonical metric  (see theorem
\eqref{y1}).\\

 Let $L$ be a line bundle on $\p^1(\C)$. A metric 
$\vc$ on $L$ is semipositive if it is the uniform limit
of a sequence of semipositive smooth metrics. 
The metric $\vc$ is DSP if it is
the quotient of two semipositives ones. We denote by
$\mathcal{M}_\Q$ the set of places of $\Q$. For 
any $v\in \mathcal{M}_\Q$, we denote by 
$\p^{1, an}_v$ the $v$-adic analytification of $\p^1_\Q$.
Similarly  a line bundle $L$ on $\p^1_\Q$ defines 
a collection of analytic line bundles $\{L_v^{an}\}_{v\in 
\mathcal{M}_\Q}$, see   \cite[\S 3]{Burgos3} 
for more details.

\begin{definition}\label{z3}
A metrized divisor on $\p^1_\Q$ is a pair 
$\overline{D}=(D,(\vc_v)_{v\in \mathcal{M}_\Q})$ formed by a
divisor $D$ with  $\vc_\infty $ is 
a continuous hermitian
 metric  on $\mathcal{O}(D)_\infty$ and
$\vc_v$ is the canonical metric of $\mathcal{O}(D)_v$ 
for $v$ a non-archimedean place. We say that
 $\overline{D}$ is
smooth or semipositive if so is the metric $\vc_\infty
$.  We say 
that $\overline{D}$ is a DSP divisor if it is
 the difference of two semipositive divisors.
 The Green function
of $\overline{D}$ is the function $g_{\overline{D}}:
\p^1(\C)\setminus |D|
\rightarrow \R$ given by
\[
g_{\overline{D}}(p)=-\log \|s_D(p)\|_\infty,
\]
where $s_D$ is the canonical section of $\mathcal{O}(D)$.
\end{definition}

Let $\overline{D}$ be a metrized DSP
 divisor on $\p^1_\Q$ 
as in \eqref{z3}.  We suppose that $\overline{D}$ is
toric, see \cite[\S.4]{Burgos3}. This means that
$D$ is a toric divisor and 
  $\vc_\infty$ is invariant
under the action of $\s$ the compact torus of $\p^1(\C)$ 
(see \cite[definition 4.12]{Burgos3} and 
\cite[proposition 4.16]{Burgos3}). In the sequel,  we assume 
that $\overline{D}$ satisfies these hypothesis and 
$D$ is nef.

\begin{theorem}\label{y1} Under the previous hypothesis, 
we have \[
h_{\overline{D}}(\p^1_\Q)\leq \widehat{\mathrm{vol}}_\chi(\p^1_\Q,\overline{D}).
\]

\end{theorem}

In order to prove this theorem, we assume first that $\overline{D}$ is smooth. By definition,
$g_{\overline{D}}$ is
 a smooth weight of $\vc_\infty$. We denote by $P g_{\overline{D}}$ the equilibrium weight of $g_{\overline{D}}$
 (see the appendix) instead of $P_{\p^1} g_{\overline{D}}$ and by
 $\vc_P$ the hermitian metric
defined by $P g_{\overline{D}}$ and we denote by
 $\overline{D}_P$ the metrized divisor $D$ endowed
 with the metric $\vc_P$ on the archimedean place and
 with the canonical metric on all non-archimedean
 places.
\begin{claim}
$\overline{D}_P$ is a semipositive toric divisor.
\end{claim}
\begin{proof}
By definition $Pg_{\overline{D}}$ is a psh weight on $\mathcal{O}(D)_\infty$ and we know that $\vc_P$ is a continuous 
metric (see for instance \cite[\S 1.4, before (1.8)]{Berman}). Then the Chern current
$c_1((\mathcal{O}(D),P\vc))$ is semipositive\footnote{that is $c_1((\mathcal{O}(D),P_X\vc))\geq 0$}. By
\cite[theorem 4.6.1]{Maillot}\footnote{Notice that a semipositive metric as in \eqref{z3}
corresponds to the notion of admissible metric in \cite{Zhang} and in \cite[4.5.5]{Maillot}}, $\vc_P$ is a semipositive
metric.

 Let
$g$ be a psh weight function on $\mathcal{O}(D)_\infty$ with $g\leq g_{\overline{D}}$. Let $\theta\in \mathbb{S}^1$.
 We set
$g_\theta$ the  function given by $g_\theta(z)=g(\theta\cdot z)$ for any
$z\in \p^1(\C)$.
 Then $g_\theta$ is clearly a psh weight on $\mathcal{O}(D)_\infty$. We have $g_\theta(z)=g(\theta\cdot z)\leq g_{\overline{D}}(\theta\cdot
z)=g_{\overline{D}}(z)$, $\forall z\in \p^1(\C)$.
Then, $g_\theta(z)\leq Pg_{\overline{D}}(z)$, $\forall z\in \p^1(\C)$. Therefore,
$Pg_{\overline{D}}
(\theta\cdot z)\leq Pg_{\overline{D}}(z)$, $\forall\, \theta\in
\mathbb{S}^1, \forall\, z\in \p^1(\C)$. We conclude that
\[
Pg_{\overline{D}}(\theta\cdot z)=Pg_{\overline{D}}
(z)\quad\forall\,\theta\in \mathbb{S}^1,\,\forall \,z\in \p^1(\C).
\]
Which means  that $\vc_P$ is an invariant metric. We conclude that $\overline{D}_P$ is a  semipositive toric divisor
on $\p^1_\Q$.
\end{proof}

Recall that if $\overline{D}':=
(D, (\vc'_v)_{v\in \mathcal{M}_\Q})$ is a smooth metrized 
divisor as in \eqref{z3}, then by 
\cite[proposition 3.2.2]{BoGS}, we have 
\[
h_{\overline{D}}(\p^1_\Q)-h_{\overline{D}'}(\p^1_\Q)
=-\int_X\bigl(g_{\overline{D}}-g_{\overline{D}'}\bigr)\bigl(
c_1(\mathcal{O}(D), \vc)
+c_1 (\mathcal{O}(D),\vc')\bigr).\]
By \cite{Zhang}, one can  extend this equality 
to the case of  DSP divisor $\overline{D}'$,  and we have 
\[
h_{\overline{D}}(\p^1_\Q)-h_{\overline{D}'}(\p^1_\Q)
=-\int_X\bigl(g_{\overline{D}}-g_{\overline{D}'}\bigr)\bigl(
c_1(\mathcal{O}(D), \vc)
+c_1 (\mathcal{O}(D),\vc')\bigr),\]
where $c_1 (\mathcal{O}(D),\vc')$ is the first 
Chern current
of $(\mathcal{O}(D),\vc')$.\\

Since $\overline{D}_P$ is semipositive, then the previous
equality gives
\[
h_{\overline{D}}(\p^1_\Q)-h_{\overline{D}_P}(\p^1_\Q)=-\int_X\bigl(g_{\overline{D}}-g_{\overline{D}_P}\bigr)\bigl(
c_1(\mathcal{O}(D), \vc)
+c_1 (\mathcal{O}(D),\vc_P)\bigr).
\]
From \eqref{z2}, we have
\begin{equation}\label{z4}
h_{\overline{D}}(\p^1_\Q)\leq h_{\overline{D}_P}(\p^1_\Q).
\end{equation}
Since $\overline{D}_P$ is a semipositive toric divisor, then by \cite[corollary 5.8]{Burgos3}
\begin{equation}\label{z6}
h_{\overline{D}_P}(\p^1_\Q)=\widehat{\mathrm{vol}}_\chi(\p^1_\Q,\overline{D}_P),
\end{equation}
and by \cite[theorem 5.6]{Burgos3}, we have
\[
\widehat{\mathrm{vol}}_\chi(\p^1_\Q,\overline{D}_P)=2\int_{\Delta_D}
\vartheta_{\overline{D}_P}d\mathrm{vol}_\R,
\]
where $\vartheta_{\overline{D}_P}$ is the roof function
associated to $\overline{D}_P$ (see \cite[definition 4.17]{Burgos3}).

\begin{claim}\label{z7}
We have,
\[
\vartheta_{\overline{D}_P}=\vartheta_{\overline{D}},
\]
on $\Delta_D$.
\end{claim}
\begin{proof}
This is an easy consequence of the  combination
 of \cite[proposition 5.1 (1)]{Burgos3} and
\cite[proposition 2.8]{BermanBoucksom}. Indeed, by \cite[proposition 2.8]{BermanBoucksom} we have 
$\sup_{\p^1}\|s\|_{k\overline{D}}=\sup_{\p^1}\|s\|_{k\overline{D}_P}$ for any $s\in H^0(\p^1,\mathcal{O}(kD))$ and $k\in
\N^\ast$. But,
we know that $\sup_{\p^1}\|s_m\|=\exp(-k\vartheta_{\overline{D}}(\frac{m}{k}))$ where $s_m$ is the global
section of $\mathcal{O}(kD)$ corresponding to $m\in k\Delta_D\cap \Z$ (see for instance
 \cite[proposition 5.1]{Burgos3})). By continuity and density arguments we deduce the equality of the claim.
\end{proof}
By \cite[theorem 5.6]{Burgos3} and the claim \eqref{z7} we have,
\begin{equation}\label{z5}
\widehat{\mathrm{vol}}_\chi(\p^1_\Q,\overline{D}_P)=
\widehat{\mathrm{vol}}_\chi(\p^1_\Q,\overline{D}).
\end{equation}
Then from \eqref{z4}, \eqref{z6} and \eqref{z5} we conclude that
\[
h_{\overline{D}}(\p^1_\Q)\leq \widehat{\mathrm{vol}}_\chi(\p^1_\Q,\overline{D}).
\]
Thus we prove the  theorem \eqref{y1} for $\overline{D}$ smooth. Now let $\overline{D}$ be a  toric $DSP$ divisor.
By defintion,
 there exist
$\overline{D}_1=(D_1,\vc_1)$ and $\overline{D}_2=(D_2,\vc_2)$ two semipositive toric divisors such that
$D=D_1-D_2$ and $\vc=\vc_1\otimes \vc_2^{-1}$. For $i=1,2$, we choose $(\vc_{i,n})_{n\in \N}$ a sequence of smooth and
 semipositives metrics\footnote{Here semipositive, means that the associated first Chern form is semipositive}
   on $\overline{D}_i$
 converging uniformly to $\vc_i$. We set $\vc_n:=\vc_{1,n}\otimes \vc_{2,n}^{-1}$ and $\overline{D}_n:=(D,\vc_n)$ for
 any $n\in \N$.
  This is a sequence of smooth metrics
on $\mathcal{O}(D)$ converging uniformly to $\vc$.  The smooth case implies that
\begin{equation}\label{y2}
h_{\overline{D}_n}(\p^1_\Q)\leq \widehat{\mathrm{vol}}_\chi(\p^1_\Q,\overline{D}_n)\quad \forall\,n\in \N.
\end{equation}
By \cite{Zhang}, we  have that the LHS of \eqref{y2}  converges to
$h_{\overline{D}}(\p^1_\Q)$. Moreover, we can establish  that the roof functions
of a sequence of metrics converging uniformly
form a sequence of continuous functions on $\Delta_D$ converging uniformly. We deduce that the RHS of
\eqref{y2} converges
to  $\int_{\Delta_D}
\vartheta_{\overline{D}}d\mathrm{vol}_\R$ that is to  $\widehat{\mathrm{vol}}_\chi(\p^1_\Q,\overline{D})$, by
\cite[theorem 5.6]{Burgos3}.
We conclude that
\[
h_{\overline{D}}(\p^1_\Q)\leq \widehat{\mathrm{vol}}_\chi(\p^1_\Q,\overline{D}).
\]

This ends the proof of the theorem \eqref{y1}.

\section{Appendix}
Let $X$ be
compact manifold  of dimension $n$ and $L$  an ample
 holomorphic line bundle on  $X$.
Let $\phi$ be a weight
of a continuous hermitian metric $e^{-\phi}$
on $L$. When $\phi$ is smooth we define the Monge-Amp\`ere operator as
\[
\mathrm{MA}(\phi):=(dd^c\phi)^{\wedge n}.
\]
The equilibrium weight of $\phi$ is defined as:
\[
P_X\phi:=\sup{}^\ast\{\psi\,\text{psh weight on L},\, \psi\leq \phi\,\text{on}\, X \}
\]
where  $\ast$ denotes upper semicontinuous regularization. When $\phi$ is smooth then
$P_X\phi=\sup \{\psi\,\text{psh weight on}\, L, \, \psi\leq \phi\,\text{on}\, X \}$. It is known that
 $P_X\phi$
is a psh weight and the metric $e^{-P_X\phi}$ is continuous (see for instance \cite[\S 1.4, before (1.8)]{Berman}).
 By the theory of Bedford-Taylor,
 the Monge-Amp\`ere operator can be extended to locally bounded psh
weights $\phi$ (see \cite{Bedford}).\\

By \cite[proposition 2.10]{BermanBoucksom} we have
\begin{equation}\label{z1}
\int_X(P_X\phi-\phi)\mathrm{MA}(P_X\phi)=0.
\end{equation}

When $\dim(X)=1$, we have
\begin{equation}\label{z2}
\int_X(\phi-P_X\phi)(
\mathrm{MA}
(\phi)+\mathrm{MA} (P_X\phi))
\leq 0
\end{equation}
Indeed,
\begin{align*}
\frac{1}{2}\int_X(\phi-P_X\phi)(
\mathrm{MA}
(\phi)+\mathrm{MA} (P_X\phi))
&=\frac{1}{2}\int_X(\phi-P_X\phi)(dd^c \phi-dd^c P_X\phi)\quad\text{by}\,
\eqref{z1}\\
&=-\int_Xd(\phi-P_X\phi)\wedge d^c(\phi-P_X\phi)\\
&\leq 0.
\end{align*}

\bibliographystyle{plain}

\bibliography{biblio}

\end{document}